\documentclass[12pt]{article}

\usepackage[paper=letterpaper,margin=0.7in,twoside=false,includehead]{geometry}

\usepackage{amsfonts,amsmath,amsthm, amssymb, thmtools, thm-restate} 
 
\usepackage{hyperref}
\usepackage[capitalize]{cleveref}

\usepackage{graphicx}
\usepackage{fancyhdr}
\usepackage{xparse}
\usepackage{mathtools}
\usepackage{float}
\usepackage{adjustbox}

\usepackage[shortlabels]{enumitem}
\usepackage[normalem]{ulem}

\SetEnumitemKey{thmparts}{topsep=2pt plus 1pt minus1pt, itemsep=1pt plus0.5pt minus0.5pt}

\usepackage{xcolor}
\usepackage{verbatim}
\usepackage{fancyvrb}
\usepackage{fvextra}

\usepackage[T1]{fontenc}

\usepackage{tikz}
\usetikzlibrary{arrows,petri,topaths}
\usepackage{tkz-berge}
\usepackage[vcentermath]{youngtab} 
\usepackage{young}
\usetikzlibrary{decorations.pathreplacing}
\usetikzlibrary{arrows,shapes, positioning}
\usetikzlibrary{calc}
\usetikzlibrary{graphs, graphs.standard}
\usepackage{wrapfig}

\newcommand{\N}{\mathbb{N}}

\newcommand{\C}{\mathbb{C}}

\DeclarePairedDelimiterX\setof[2]{\{}{\}}{#1\,:\,#2}

\DeclareDocumentCommand{\shad}{O{q}}{\partial_{#1}}
\DeclareDocumentCommand{\upshad}{O{t}}{U^{#1}}

\DeclareDocumentCommand{\k}{O{t}}{k^{#1}}
\DeclareDocumentCommand{\layer}{O{s}O{\N}}{\binom{#2}{#1}}
\DeclareDocumentCommand{\nlayer}{O{n} m}{\binom{[{#1}]}{#2}}
\DeclareDocumentCommand{\C}{O{s}}{\mathcal{C}_{#1}}

\DeclareDocumentCommand{\iseq}{m m}{n_{#1},n_{{#1}-1},\dots,n_{{#1}-{#2}+1}}
\DeclareDocumentCommand{\is}{O{s}O{s}O{\ell}}{[\iseq{#2}{#3}]_{#1}}
\DeclareDocumentCommand{\seq}{O{s}O{\ell}}{(\iseq{#1}{#2})}

\newtheorem{thm}{Theorem}[section]
\newtheorem{lem}[thm]{Lemma}
\newtheorem{cor}[thm]{Corollary}

\theoremstyle{definition}
\newtheorem{defn}[thm]{Definition}
\newtheorem{qu}{Question}

\newtheorem{rem}[thm]{Remark}
\newtheorem{obs}[thm]{Observation}
\newtheorem{example}[thm]{Example}

\numberwithin{case}{thm}

\numberwithin{subcase}{case}
\newtheorem{prop}[thm]{Proposition}
\newtheorem{porism}[thm]{Porism}

\newtheorem{construction}[thm]{Construction}

\newcommand{\diam}{\mathrm{diam}}

\begin{document}

\pagestyle{plain}

\title{Radio gracefulness of Moore graphs and beyond}
\author{An Cao, Aleyah Dawkins, Julian Hutchins, Orlando Luce}
\maketitle

\abstract{
The study of radio graceful labelings is motivated by modeling efficient frequency assignment to radio towers, cellular towers, and satellite networks. For a simple, connected graph $G = (V(G), E(G))$, a radio labeling is a mapping $f: V(G) \rightarrow \mathbb{Z}^+$ satisfying (for any distinct vertices $u,v$)\begin{align*}
        |f(u)-f(v)| + d(u,v) &\geq \diam(G)+1,
\end{align*} where $d(u,v)$ is the distance between $u$ and $v$ in $G$ and $\diam(G)$ is the diameter of $G$. A graph is radio graceful if there is a radio labeling such that $f(V(G)) = \{1, \dots, |V(G)|\}$.
In this paper, we determine the radio gracefulness of low-diameter graphs with connections to high-performance computing, including Moore graphs, bipartite Moore graphs, and approximate Moore graphs like $(r,g)-$cages, Erd\H{o}s-Rényi polarity graphs, and McKay-Miller-\v{S}irá\v{n} graphs. We prove a new necessary and sufficient condition for radio graceful bipartite graphs with diameter $3$. We compute the radio number of $(r,g)-$cages arising from generalized $n-$gons. Additionally, we determine Erd\H{o}s-Rényi polarity graphs and McKay-Miller-\v{S}irá\v{n} graphs are radio graceful.
}

\section{Introduction}\label{sec:intro}

    Graph labeling is an integral topic of study in graph theory, with relations to problems in graph coloring, network design, and communication theory. A graph labeling is a function that assigns integers to the vertices and/or edges of a graph and that satisfies certain conditions. The labeling conditions often model real world problems, and the study of graph labeling reveals useful combinatorial structures or optimized assignments for practical applications.

    One such application is \emph{frequency assignment problem} (or \emph{channel assignment problem}). First introduced by Hale in 1980 \cite{halechannelassign}, it addresses the challenge of assigning frequencies to transmitters such that interference is minimized. This foundational problem inspired the radio labeling problem formulated by Chartrand et al. in 2001 \cite{chartrand2001radiolabeling}, as well as numerous variants \cite{L21, Lhk, PONRAJ2015224}, which has modern applications to satellite and wireless communication networks. In this setting, the graph represents a network of radio stations, with edges modeling geographic proximity, which can potentially lead to signal interference. The goal is to assign distinct frequencies to each station, such that stations that are closer in the network (i.e., have shorter distance in the graph) are assigned more widely separated frequencies. For a simple, connected graph $G = (V(G), E(G))$, a radio labeling can then be defined as a mapping $f: V(G) \rightarrow \mathbb{Z}^+$ satisfying \begin{align*}
        |f(u)-f(v)| + d(u,v) &\geq \diam(G)+1,
    \end{align*} where $d(u,v)$ is the distance between $u$ and $v$ in $G$ and $\diam(G)$ is the diameter of $G$. 

    A secondary objective of the frequency assignment problem is the minimization of bandwidth. This optimization corresponds to finding a labeling that minimizes the span, or maximum integer assigned as a label. We call a radio labeling \emph{graceful} if the difference between the largest and smallest assigned labels is minimal, i.e. $|V(G)|-1$. Previous research has investigated the radio gracefulness of some common families of graphs \cite{hamming}, their radio numbers \cite{multilevel-path-cycle,cartesianprod-rn,Zhang-cycle-radio}, and more general characterizations of radio graceful graphs \cite{chartrand2001radiolabeling,saha2020}.

Motivated by recent advances in high-performance computing network design, we are interested in studying efficiently structured networks, or low-latency, high-bandwidth, and scalable networks, and their applicability to the frequency assignment problem with the existence of a radio graceful labeling for such networks. In extremal graph theory, these networks relate to the degree-diameter problem and the degree-girth problem. The degree-diameter problem seeks to determine the largest possible number of vertices in a graph of given maximum degree $\Delta$ and diameter $D$. An upper bound on the order of such graphs is called the Moore bound $M(\Delta, D)$, and is given by the formula:
\begin{align*}
    M(\Delta, D) = 1 + \Delta\sum_{i = 0}^{D-1}(\Delta-1)^i. 
\end{align*}

Graphs that attain this bound are known as Moore graphs. Moore graphs exhibit a highly symmetric structure, including being distance-regular and vertex-transitive. Vertex-transitivity guarantees self-centeredness, a necessary condition for radio graceful graphs. Highly symmetric graphs are also often self-centered. All possible Moore graphs are summarized in \cref{table1} below \cite{moore-miller}.

\begin{table}[H]
    \centering
    \begin{tabular}{|c|c|c|}
        \hline
        Graph & Moore graph & Cage \\
         & $(\Delta, D)$ & $(r,g)$ \\\hline
         Complete graphs $K_n$& $(n-1,1)$& $(n-1,3)$ \\\hline
         Odd cycles $C_{2n+1}$& $(2,n)$& $(2,2n+1)$ \\\hline
         Petersen graph & $(3,2)$& $(3,5)$ \\\hline
         Hoffman-Singleton graph & $(7,2)$& $(7,5)$ \\\hline
        Hypothetical graph of diameter 2, girth 5, degree 57 & $(57,2)$& $(57,5)$ \\\hline
    \end{tabular}
    \caption{List of all possible Moore graphs}
    \label{table1}
\end{table}

A closely related problem is the degree-girth problem, which is to find the smallest possible regular graph of a given degree $r$ and girth $g$. Such graphs are called $(r,g)-$cages. Notably, the Moore bound $M(\Delta, D)$ exactly provides the lower bound on the order of a $(\Delta,2D+1)-$cage. As a result, every Moore graph is also a cage, though not every cage attains the Moore bound.

While all classical Moore graphs have odd girth, there is a generalized definition for even girth, that is, the bipartite Moore graphs, or generalized Moore graphs, which are the largest possible bipartite graphs of given maximum degree $\Delta$ and diameter $D$. The bipartite Moore bound $M_b(\Delta, D)$ given by

$$ M_b(\Delta, D) = 2\sum_{i = 0}^{D-1}(\Delta-1)^i$$ is the lower bound on the order of a cage with even girth, hence these graphs are the smallest possible graphs with the given maximum degree $\Delta$ and girth $g=2D$, that is $(\Delta,2D)-$cages. As seen in \cref{table2}, the bipartite Moore graphs with diameters $ 3, 4, 6$ correspond to the incidence graphs of projective planes of order $q$, generalized quadrangles of order $q$, and generalized hexagons of order $q$, respectively. Such bipartite Moore graphs are only known to exist if $q$ is a prime power.

\begin{table}[H]
    \centering
    \begin{tabular}{|c|c|c|}
        \hline
        Graph & bi-Moore graph & Cage \\
         & $(\Delta, D)$ & $(r,g)$ \\\hline
         Complete bipartite graphs $K_{n,n}$& $(n,2)$& $(n,4)$ \\\hline
         Even cycles $C_{2n}$& $(2,n)$& $(2,2n)$ \\\hline
         Incidence graph of projective planes of order $q$ & $(q+1,3)$& $(q+1,6)$ \\\hline
         Incidence graph of generalized quadrangles of order $q$ & $(q+1,4)$& $(q+1,8)$ \\\hline
         Incidence graph of generalized hexagons of order $q$ & $(q+1,6)$& $(q+1,12)$ \\\hline
         
    \end{tabular}
    \caption{List of all possible bipartite Moore graphs}
    \label{table2}
\end{table}

Other interesting families of graphs related to the Moore bound, and specifically the Moore bound for diameter 2, are Erd\H{o}s-Rényi polarity graphs \cite{ER_graph} (or Brown graphs \cite{Brown_1966}) and McKay-Miller-\v{S}irá\v{n} graphs (MMS graphs) \cite{MMS}. For odd prime powers $q \ge 7$, the polarity graphs give the currently largest known order of graphs with diameter $2$ and maximum degree $q+1$. In particular, a polarity graph with maximum degree $q+1$ has $q^2+q+1$ vertices, which is only $d$ less than the Moore bound. Modifications of these graphs have been constructed to show that the Moore bound can be approached asymptotically, even if they cannot be reached \cite{siran}. The MMS graphs are an infinite class of vertex-transitive graphs with diameter $2$ closest to reaching the Moore bound \cite{MMS}. These constructions, whose sizes are ``close'' to the Moore bound, are called approximate Moore graphs \cite{approximateMoore}. In particular, polarity graphs are constant additive approximations, while MMS graphs are constant multiplicative approximations.

In this paper, we consider the graphs outlined above. \cref{sec:prelim} introduces definitions and proof methods used throughout the paper. In \cref{sec:main}, we discuss in more detail bipartite graphs, bipartite Moore graphs corresponding to $(r,g)-$cages, and approximate Moore graphs, including Erd\H{o}s-Rényi polarity graphs and MMS graphs. Conditions for a bipartite graph to be radio graceful are given, as well as a necessary and sufficient condition for diameter $3$ bipartite graphs to be radio graceful. In addition, the radio number of a bipartite Moore graph corresponding to a $(r,g)-$ cage is determined, and the radio gracefulness of Erd\H{o}s-Rényi polarity graphs and MMS graphs is determined. We conclude in \cref{sec:open} with a question on the radio gracefulness of all $(r,g)-$cages.

\section{Preliminaries}\label{sec:prelim}

    Given a graph $G$, a \emph{path} in $G$ is a sequence of adjacent vertices without repeat. A sequence of adjacent vertices where only the initial and final vertex are the same is a \emph{cycle}. A \emph{Hamiltonian path (cycle)} in $G$ is a path (cycle) that visits every vertex of $G$ exactly once. If $G$ contains a Hamiltonian path, it is \emph{traceable}. If it contains a Hamiltonian cycle, it is \emph{hamiltonian}. The \emph{distance} between two vertices $u$ and $v$ is the length of the shortest path between them. We denote this as $d(u,v)$ (or $d_G(u,v)$ when specifying the graph $G$ is necessary). When $d(u,v) = 1$, we write $u \sim v$ for convenience. The \emph{diameter} of $G$ is denoted and defined as $\diam(G) = \{\max\{d(u,v)\} \ | \ u,v \in V(G)\}$. The \emph{girth} of a graph is the length of the smallest cycle contained within it. Throughout this paper, we take all graphs $G$ to be simple. We also write $r$ to denote the degree of a regular graph.

    \begin{defn}
        The \emph{antipodal graph} of a graph $G$, denoted as $A(G)$, is a graph on the same set of vertices as G where each pair of vertices $u$ and $v$ are connected in $A(G)$ if and only if $d_G(u,v) = \diam(G)$. For $\diam(G)=2$, we have $A(G)=\overline{G}$.
    \end{defn}

    We now recall the definition of a radio labeling along with its natural refinement.

    \begin{defn}
    Given a simple connected graph $G$, a labeling $f: V(G) \to \mathbb{Z}^+$ is a \emph{radio labeling} if it satisfies the inequality \begin{align}\label{radio_condition}
        |f(u)-f(v)| + d(u,v) &\geq \text{diam}(G)+1
    \end{align}for all distinct vertices $u$ and $v$. 
    \end{defn}
    The \emph{span} of a labeling $f$ is the largest label assigned to a vertex. The \emph{radio number} of $G$ is $rn(G) = \min\{span(f) \ | \ f \text{ is a radio labeling of $G$}\}$. A radio labeling is \emph{graceful} if $span(f) = |V(G)|$. A graph $G$ is \emph{radio graceful} if a radio graceful labeling exists, or equivalently, if $rn(G) = |V(G)|$.

    The following result proves useful for quickly determining if a graph is not radio graceful.
    
    \begin{thm}[Saha and Basunia \cite{saha2020}]
        \label{Antipodal}
        If a graph $G$ is radio graceful, then $A(G)$ has a Hamiltonian path.
    \end{thm}

    When restricting to graphs with diameter $2$, Chartrand, Erwin, Harary, and Zhang were able to give a necessary and sufficient condition for radio gracefulness that proves very helpful below.
    
    \begin{thm}[Chartrand, Erwin, Harary, and Zhang \cite{chartrand2001radiolabeling}]
        \label{Antipodal_2}
        Let $G$ be a graph with diameter $2$. Then $G$ is radio graceful if and only if $A(G)$ has a Hamiltonian path.
    \end{thm}

    The requirement of a Hamiltonian path naturally leads to the use of the following results. 

    \begin{thm}[Dirac \cite{Dirac}]
    \label{Dirac}
        If the minimum degree of a graph $G$ on n vertices is at least $\frac{n-1}{2}$ then $G$ has a Hamiltonian path. 
    \end{thm}

    \begin{thm}[Corollary of Moon and Moser \cite{Moon1963}]\label{reg-bip-ham}
    If $G$ is an $r-$regular bipartite graph $G = (U,V,E)$, where $|U| = |V| = n < 2r$, then $G$ has a Hamiltonian path.
    \end{thm}

    \subsection{Generalized polygons}

    Many of the graphs considered in this paper arise from \emph{generalized polygons} (or \emph{generalized $n-$gons}).

    \begin{defn}
        Let $\mathcal{I}$ be an ordered triple $(P,L,I)$ where $P$ is the nonempty \emph{set of points} $p$ of $\mathcal{I}$, $L$ is the nonempty \emph{set of lines} $l$ of $\mathcal{I}$, and $I \subseteq P \times L$ is the \emph{incidence relation}. Let $G_I$ be the associated bipartite \emph{incidence graph} on $P \cup L$ with edges joining the points of $P$ to their incident lines in $L$. Then $\mathcal{I}$ is a \emph{generalized $n-$gon} if the following four conditions are satisfied:
        \begin{enumerate}
        \item There exist $s \ge 1$ and $t \ge 1$ such that every line is incident to exactly $s+1$ points and every point is incident to exactly $t+1$ lines.
        \item Any two distinct lines intersect in at most one point and there is at most one line through any two distinct points.
        \item  The diameter of the incidence graph $G_I$ is $n$.
        \item The girth of $G_I$ is $2n$.
        \end{enumerate}
    \end{defn}

    The generalized $n-$gon $\mathcal{I}$ has order $(s,t)$, or order $s$ in the case $s=t$. When $n = 3, 4, 6$, generalized $n-$gons correspond to finite projective planes, generalized quadrangles, and generalized hexagons, respectively \cite{generalized-polygons}. We now discuss each in detail, introducing the properties that are unique to each, beginning with $n=3$. In this case, it is known $s=t$ with the existence of the finite projective plane known when $s$ is a prime power $q$.
    
    \begin{defn}
        A \emph{projective plane} $\mathbb{P}$ of order $q$ is an ordered triple $(P,L,I)$ where $P$ is the \emph{set of points} $p$ of $\mathbb{P}$, $L$ is the \emph{set of lines} $l$ of $\mathbb{P}$, and $I \subseteq P \times L$ is the \emph{incidence relation} satisfying the following four properties: \begin{enumerate} 
        \item Any two points determine a line. 
        \item Any two lines determine a point. 
        \item Every point is incident with $q+1$ lines. 
        \item Every line is incident with $q+1$ points.
        \end{enumerate}
    \end{defn}

    From this, a combinatorial argument can be made to prove the following statement.

    \begin{rem}\label{incidenceGraph_plane}
        The incidence graph $G_I$ of a finite projective plane of order $q$ is a $(q+1)-$regular bipartite graph with diameter $3$, girth $6$, and parts corresponding to the set of points $P$ and the set of lines $L$ where $|P| = |L| = q^2+q+1$.
    \end{rem}

    The smallest projective plane, the Fano plane, is seen in Figure~\ref{fig:fano_cage1}. The Fano plane, arising from a field of order $2$, has seven points and seven lines, with three points on every line and three lines through every point.

    \begin{figure}[H]
        \centering
        \includegraphics[width=0.3\linewidth]{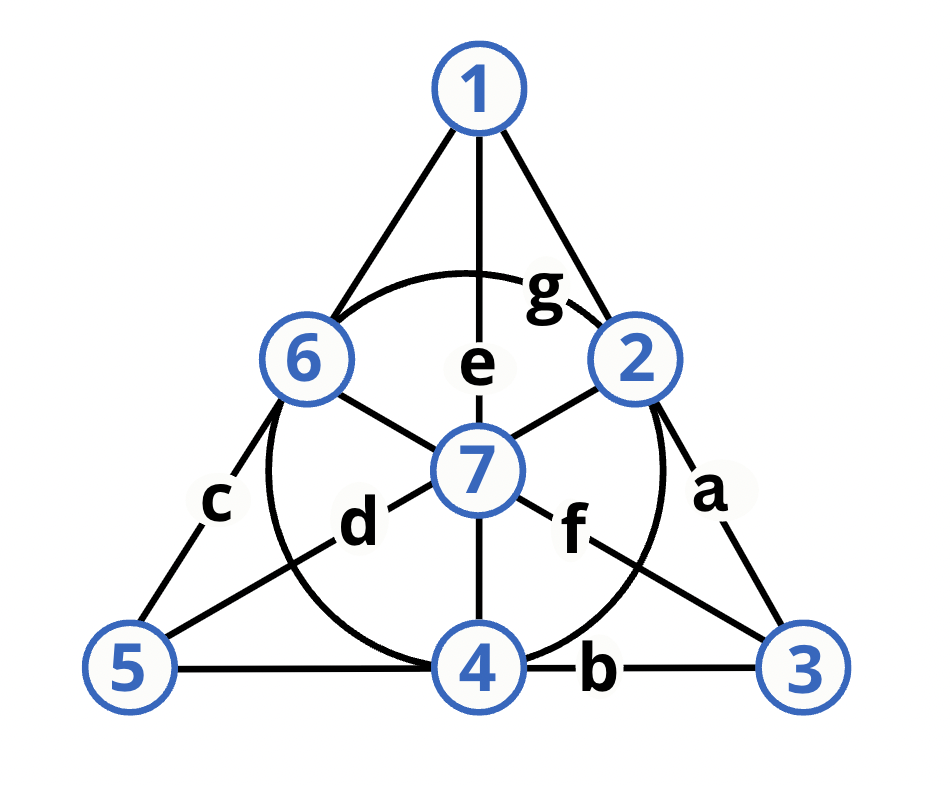}
        \caption{Fano plane}
        \label{fig:fano_cage1}
    \end{figure}

    When $q = 2$ the incidence graph $G_I$ of the Fano plane, as in Figure~\ref{fig:fano_cage2}, is the unique $(3,6)-$cage (Heawood graph) \cite{CageSurvey}. The graph $A(G_I)$ has edges between points and lines that are not incident in $\mathbb{P} \cong PG(2,2)$.

    \begin{figure}[H]
        \centering
        \includegraphics[width=0.6\linewidth]{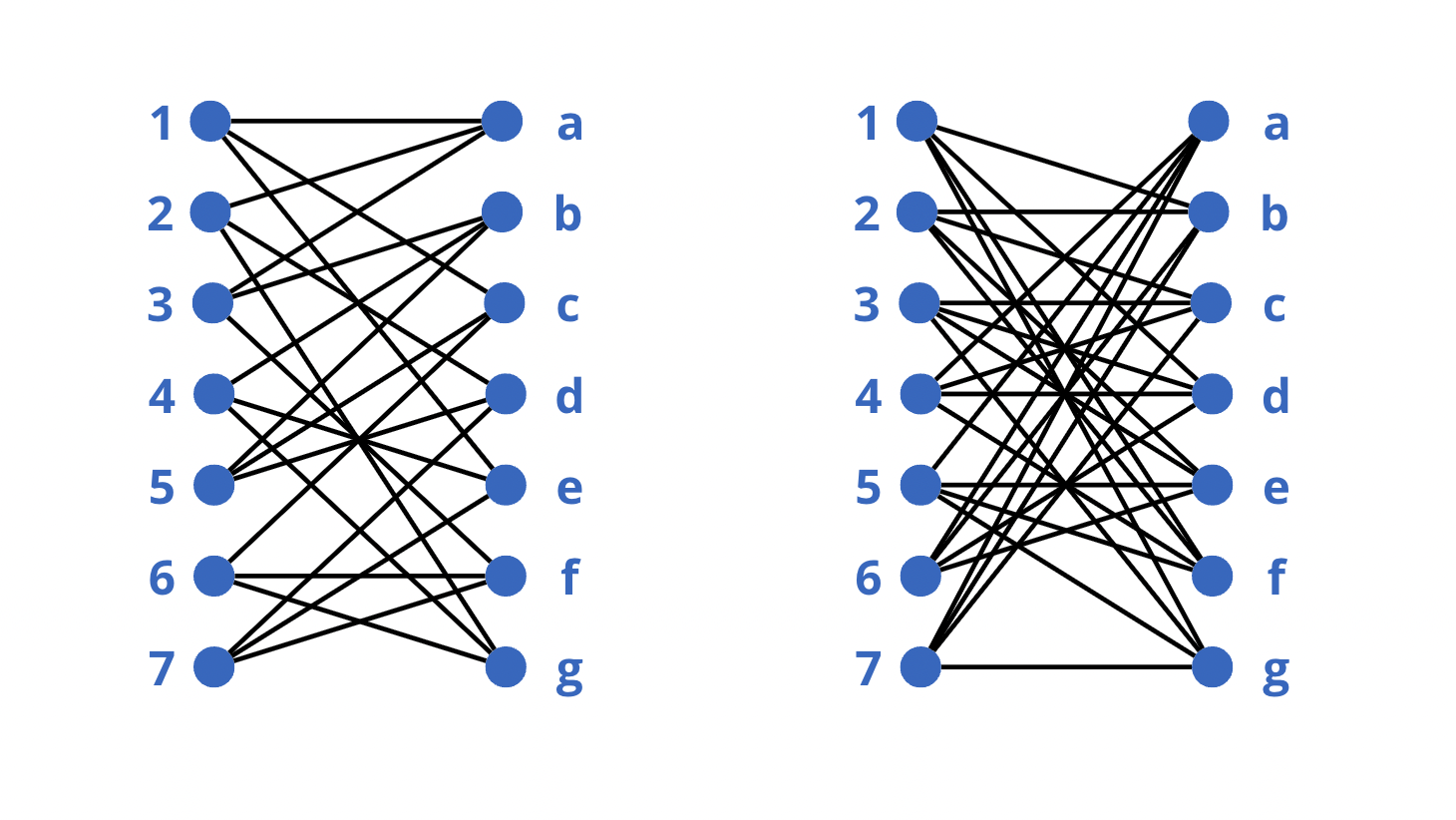}
        \caption{$G_I$ and $A(G_I)$}
        \label{fig:fano_cage2}
    \end{figure}
    
    \begin{defn}
        A \emph{generalized quadrangle} of order $(s, t)$ has $(s + 1)(st + 1)$ points and $(t + 1)(st + 1)$ lines, and satisfies the following five properties:
        \begin{enumerate} 
        \item Any two points lie on at most one line.
        \item Any two lines intersect in at most one point.
        \item Every line is incident with $s + 1$ points.
        \item Every point is incident with $t + 1$ lines.
        \item For any point $p \in P$ and line $l \in L$ where $(p,l) \notin I$, there is exactly one line incident with $p$ and intersecting $l$.
        \end{enumerate}
    \end{defn}

    \begin{rem}\label{incidenceGraph_quad}
        The incidence graph $G_I$ of a generalized quadrangle of order $q$ is a $(q+1)-$regular bipartite graph with diameter $4$, girth $8$, and parts corresponding to the set of points $P$ and the set of lines $L$ where $|P|=|L|=(q+1)(q^2+1)$.
    \end{rem}

    \begin{defn}
        A \emph{generalized hexagon} of order $(s, t)$ has $(s+1)(s^2t^2+st+1)$ points and $(t+1)(s^2t^2+st+1)$ lines, and satisfies the following five properties:
        \begin{enumerate} 
        \item Any two points lie on at most one line.
        \item Any two lines intersect in at most one point.
        \item Every line is incident with $s + 1$ points.
        \item Every point is incident with $t + 1$ lines.
        \item For any point $p \in P$ and line $l \in  L$ where $(p,l) \notin I$, there is a unique shortest path from $p$ to $l$ of length 3 or 5.
        \end{enumerate}
    \end{defn}

    \begin{rem}\label{incidenceGraph_hex}
        The incidence graph $G_I$ of a generalized hexagon of order $q$ is a $(q+1)-$regular bipartite graph with diameter $6$, girth $12$, and parts corresponding to the set of points $P$ and the set of lines $L$ where $|P|=|L|=(q^3+1)(q^2+q+1)$.
    \end{rem}

\section{Main Results}\label{sec:main}

    In this section, we present our main results. We begin with results on bipartite graphs in \cref{bip_sec}, giving conditions for these graphs to have a radio graceful labeling. In \cref{cages}, we apply these results to explore when bipartite Moore graphs, or cages arising from generalized polygons, are radio graceful. We then address the approximate Moore Graphs, Erd\H{o}s-Rényi polarity graphs and McKay-Miller-\v{S}irá\v{n}, and their complements in \cref{ER_q}. Throughout this section, let $q$ be a prime power.

\subsection{Bipartite graphs}\label{bip_sec}
    The necessary condition for the radio gracefulness of a graph $G$ outlined in \cref{Antipodal} extends to bipartite graphs. Two results follow when $\text{diam}(G)$ is even or $\text{diam}(G) = 3$. 
    
    \begin{thm}\label{bip_even}
        If a bipartite graph $G$ has even diameter then $G$ is not radio graceful.
    \end{thm}
    
    \begin{proof}
        For two adjacent vertices $u,v$ in $A(G)$, we have $d_G(u,v)$ is even. Thus, only vertices in the same part are adjacent in $A(G)$, leaving $A(G)$ disconnected. It follows from \cref{Antipodal} that $G$ is not radio graceful.
    \end{proof}

    \begin{thm}\label{TraceBip}
        Let $G$ be a bipartite graph with diameter 3. Then $G$ is radio graceful if and only if $A(G)$ is traceable.
    \end{thm}

    \begin{proof}
        In the antipodal graph $A(G)$, $ u \sim_{A(G)} v$ if and only if $d_G(u,v) = 3$. Thus, $A(G)$ is bipartite with the same parts as $G$. Assume $A(G)$ is traceable. Label the vertices of $G$ sequentially according to the Hamiltonian path $H$ in $A(G)$. The distance along the path $d_H(u,v) = |f(u)-f(v)|$. If $|f(u) - f(v)| = d_H(u,v) = 1$, then $d_G(u,v) = 3$. If $|f(u) - f(v)| = d_H(u,v) = 2$, then $u$ and $v$ are in the same part of both $A(G)$ and $G$, so $d_G(u,v) = 2$. If $|f(u) - f(v)| = d_H(u,v) \ge 3$, then $d_G(u,v) \ge 1$ is trivially satisfied. In all cases the radio condition is satisfied implying $G$ is radio graceful. From \cref{Antipodal}, the other direction holds.
    \end{proof}

    From the proof of \cref{TraceBip}, a useful observation can be made.

    \begin{obs}\label{bip_obs}
        Given a bipartite graph with an odd diameter with parts $U$ and $V$, the antipodal graph $A(G)$ is a bipartite graph with the same parts $U$ and $V$. Furthermore, if $\diam(G)=3$ then for $u \in U$ and $v \in V$ we have $u \sim_G v$ or $u \sim_{A(G)} v$, but not both.
    \end{obs}

    A natural question to ask is when $A(G)$ is traceable. Introducing the additional property of regularity to $G$ yields sufficient conditions for such, and the radio gracefulness of the graph follows.

    \begin{prop}\label{RegularBip}
        For $1 < r < n/2$, let $G$ be an $r-$regular bipartite graph on $2n$ vertices with diameter 3. Then $G$ is radio graceful.
    \end{prop}

    \begin{proof}
        Let $G$ be a graph satisfying the above conditions. By \cref{bip_obs}, the antipodal graph $A(G)$ is a regular bipartite graph with the same parts as $G$ where an edge between parts lies in $G$ or $A(G)$, but not both. For a vertex $v \in V(A(G))$, we have $deg(v) = n-r > \frac{n}{2}$. By \cref{reg-bip-ham}, the graph $A(G)$ is traceable. From \cref{TraceBip} it follows that G is radio graceful.
    \end{proof}

    Employing $\frac{n}{2}$ as a lower bound on $r$ guarantees $\text{diam}(G)=3$, but not the traceability of $A(G)$. Necessary and sufficient conditions for highly connected regular bipartite graphs to be radio graceful are provided.
    
    \begin{lem}\label{BoundBip}
    For $\frac{n}{2} < r < n$, let $G$ be a $r-$regular bipartite graph on $2n$ vertices. Then $G$ has diameter 3.
    \end{lem}

    \begin{proof}
        Suppose $G$ has parts $U$ and $V$. Since $G$ is regular, we have $|U|=|V|=n$. For any $u, u' \in U$, since $r>n/2$, $u$ and $u'$ share a neighbor in $V$ and $d(u,u') = 2$. Since $r < n$, there exists $v \in V$ such that $u \nsim v$. From this, $G$ has diameter 3.
    \end{proof}

    \begin{thm} \label{n-1Bip}
         Let $n \ge 3$ and $G$ be a bipartite graph on $2n$ vertices. If $G$ is $(n-1)-$regular, then $G$ is not radio graceful. If $G$ is $(n-2)-$regular, then $G$ is radio graceful if and only if $A(G) = C_n$.
    \end{thm}

    \begin{proof}
        Let $G$ be a regular bipartite graph on $2n$ with parts $U$ and $V$. Suppose $G$ is $(n-1)-$regular. From \cref{BoundBip}, $G$ has diameter 3. As seen in \cref{bip_obs}, for $u \in U$ and $v \in V$, we have $u \sim_{A(G)} v$ when $u \nsim_{G} v $. Thus, $A(G) \cong nK_2$. It follows from \cref{TraceBip} that $G$ is not radio graceful. Now suppose $G$ is $(n-2)-$regular. From a similar argument as before, $G$ has diameter 3 and $A(G)$ is $2-$regular. The antipodal graph $A(G)$ is disconnected if and only if $A(G) \ne C_n$. Therefore, $G$ is radio graceful if and only if $A(G) = C_n$.
    \end{proof}

    \subsection{Cages}\label{cages}

    We now consider three families of cages that arise as incidence graphs of generalized polygons. These graphs can be recognized as bipartite Moore graphs, that is, graphs that attain the Moore bound for even girth.

    \subsubsection{Cages with girth 6}
    The incidence graph $G_I$ of a projective plane of order $q$ is a $(q+1)-$regular graph on $2(q^2+q+1)$ vertices with diameter $3$ and girth $6$, which meets the Moore bound for even girth. Hence $G_I$ is a $(q+1,6)-$cage and, by \cite{SINGLETON1966306}, any $(q+1,6)-$cage can be associated with the incidence graph of a projective plane of order $q$.

\begin{thm}
    Let $G_I$ be the incidence graph of a projective plane of order $q$. Then $G_I$ is radio graceful.
\end{thm}

\begin{proof}
    From \cref{incidenceGraph_plane}, the vertices of $G_I$ correspond to the $q^2+q+1$ points and 
    $q^2+q+1$ lines of the projective plane. The graph $G_I$ has diameter $3$, and for any two vertices $u,v$
    \[d(u,v) =
        \begin{cases}
          1 & \text{if point $u$ is incident to line $v$ or point $v$ is incident to line $u$,} \\
          2 & \text{if $u,v$ are both points or are both lines}, \\
          3 & \text{if point $u$ is not incident to line $v$ or point $v$ is not incident to line $u$.}
          \end{cases}\]
    
    Consider the antipodal graph $A(G_I)$ of $G_I$. By definition, the graph $A(G_I)$ has the same vertex set as $G_I$ with its edges incident to a point $p$ and a line $l$ such that $(p,l) \notin I$. Thus, we have the following observations:
    \begin{enumerate}
        \item $A(G_I)$ is a balanced bipartite graph with parts corresponding to the set of points and the set of lines in the projective plane, each of size $q^2+q+1$.
        
        \item $A(G_I)$ is $q^2-$regular: Each point in the projective plane is incident to exactly $q+1$ lines, hence not incident to exactly $q^2+q+1 - (q+1) = q^2$ lines. Similarly, each line is not incident to exactly $q^2$ points. The degree of every vertex in $A(G)$ is $q^2$.
    \end{enumerate}
    
    As $q^2+q+1 < 2q^2$ for $q \geq 2$, by \cref{reg-bip-ham}, the observations show that $A(G_I)$ has a Hamiltonian path. It follows from \cref{TraceBip} that $G_I$ is radio graceful.
\end{proof}

\begin{rem}
    Any $(q+1,6)-$cage is in one-to-one correspondence with the incidence graph of a projective plane of order $q$ \cite{SINGLETON1966306}, so a $(q+1,6)-$cage is radio graceful for any prime power $q$. A radio graceful labeling can be constructed by assigning the vertices corresponding to points with odd numbers $1$ to $2n-1$ and vertices corresponding to lines with even numbers $2$ to $2n$ such that $i$ is not in the neighborhood of $i-1$.
\end{rem}

\subsubsection{Cages with girth 8}

The incidence graph $G_I$ of a generalized quadrangle of order $q$ is the unique $(q+1,8)-$cage \cite{generalized-polygons}, a bipartite $(q+1)-$regular graph on $2(q+1)(q^2+1)$ vertices with diameter $4$ and girth $8$. We have the following result on the radio gracefulness of $(q+1,8)-$cage.

\begin{thm}\label{quad-not-rg}
    Let $G_I$ be the incidence graph of a generalized quadrangle of order $q$. Then $G_I$ is not radio graceful.
\end{thm}

The proof of \cref{quad-not-rg} follows from \cref{bip_even} as $G_I$ is bipartite and has diameter $4$.

We use the following Lemma to show the existence of the square of a Hamiltonian cycle in each of the components of $A(G_I)$, a fact that we will find useful in determining the radio number of $G_I$.

    \begin{lem}(Fan and H\"{a}ggkvist \cite{square-ham-cycle})\label{ham_sq_exist}
        Let $G$ be a graph with $n$ vertices where $\delta(G) \geq \frac{5}{7}n$. Then $G$ contains the square of a Hamiltonian cycle.
    \end{lem}

    Let $G_P$ be the graph induced by vertices corresponding to the set of points $P$ in $A(G_I)$ and let $G_L$ be the graph induced by vertices corresponding to the set of points $L$ in $A(G_I)$.

    \begin{lem}\label{gen_quad_anti}
        The antipodal graph $A(G_I)$ of the incidence graph $G_I$ of a generalized quadrangle of order $q$ is a $q^3-$regular graph with two components $G_P$ and $G_L$. The graphs $G_P$ and $G_L$ both contain the square of a Hamiltonian cycle.
    \end{lem}

    \begin{proof}
        For $u,v \in V(G)$, we have $d(u,v)=4=\diam(G)$ if either $u$ and $v$ are noncollinear points, or $u$ and $v$ are nonconcurrent lines. Each point $p$ in a generalized quadrangle is collinear with $q(q+1)$ other points, since each point lies on $q+1$ lines and each of these lines is incident with $q$ points other than $p$. Therefore, for each point $p$, there are $(q+1)(q^2+1)-1-q(q+1) = q^3$ points that are not collinear with $p$. The same argument applies to lines, so $A(G_I)$ is a $q^3-$regular graph with two components corresponding to the set of points $P$ and the set of lines $L$. Since $q^3 \geq \frac{5}{7}(q+1)(q^2+1)$ for $q > 3$, \cref{ham_sq_exist} shows the existence of the square of a Hamiltonian cycle both in $G_P$ and $G_L$. 

        For $q=2$ and $q=3$, we can find the square of a Hamiltonian cycle in each bipartite part using edge lists available at \url{https://aeb.win.tue.nl/graphs/cages/cages.html}. When $q=2$, let the parts of $G$ be labeled (0, 1, $\dots$, 14) and (15, 16, $\dots$, 29). The square of a Hamiltonian cycle in each part is (0, 1, 2, 4, 5, 6, 7, 14, 13, 12, 3, 8, 11, 9, 10, 0) and (15, 19, 23, 25, 16, 18, 29, 22, 20, 26, 21, 17, 24, 28, 27, 15). When $q=3$,  let the parts of $G$ be labeled (0, 1, $\dots$, 39) and (39, 40, $\dots$, 79). The square of a Hamiltonian cycle in each part is (0, 1, 2, 3, 5, 7, 8, 6, 9, 10, 12, 11, 13, 22, 4, 14, 16, 17, 15, 18, 19, 21, 20, 24, 26, 25, 23, 27, 28, 31, 29, 30, 33, 34, 35, 32, 36, 37, 39, 38, 0) and (40, 45, 50, 43, 44, 49, 42, 47, 48, 41, 46, 51, 52, 57, 59, 55, 56, 61, 64, 54, 58, 60, 53, 62, 67, 69, 72, 65, 63, 70, 71, 79, 75, 66, 77, 76, 68, 78, 73, 74, 40).
    \end{proof}

We now compute the radio number of a $(q+1,8)-$cage.

\begin{thm}\label{rn-quad}
    Let $G$ be the $(q+1,8)-$cage. Then $rn(G) = 2(q+1)(q^2+1)+1$.
\end{thm}

\begin{proof}
    Since $G$ is the incidence graph of a generalized quadrangle of order $q$, we have the natural partition of its vertices to parts $P$ and $L$, corresponding to the set of points and the set of lines, respectively. Let $n=|P|=|L| = (q+1)(q^2+1)$. 

    From \cref{quad-not-rg} we have $G$ is not radio graceful so $rn(G) \ge 2n+1$. We give a radio labeling $f: V(G) \rightarrow \{1,2,\dots, 2n+1\}$, showing this lower bound can be achieved. We first show a consecutive labeling exists for both $P$ and $L$, then verify the existence of a gluing strategy, all of which satisfy the radio labeling condition (see \cref{radio_condition}).
 
    Consider the subgraph $G_P$ of $A(G)$ induced by the vertex set $P = \{p_1, p_2, \dots, p_n\}$. By \cref{incidenceGraph_quad}, the distance between vertices of $P$ can only be $2$ or $4$, so the radio labeling condition is satisfied only when $d(p_i,p_{i+1}) = 4$ for $i = 1,2,\dots,n-1$ and $d(p_i,p_{i+2}) = 4$ for $i = 1,2,\dots,n-2$, which can be rewritten as $p_{i+1}$ is adjacent to $p_i$ and $p_{i+2}$ in the antipodal graph $A(G)$. Similarly for the subgraph $G_L$ of $A(G)$ induced by the vertex set $L = \{l_1, l_2, \dots, l_n\}$. Then our goal is to find a Hamiltonian path in both $G_P$ and $G_L$ such that any vertex is adjacent to every vertex within distance $2$ along the path, that is, the square of the Hamiltonian path. By \cref{gen_quad_anti}, both $G_P$ and $G_L$ contain the square of a Hamiltonian cycle, say $(p_1,p_2,\dots,p_n,p_1)$ and $(l_1,l_2,\dots,l_n,l_1)$, respectively. Thus, there exists a consecutive labeling for both $P$ and $L$ that satisfies the radio labeling condition. 
    
    In the graph $G$, consider the vertices $p_{n-1}$ and $p_n$. We will show that there exists an index $t \in \{1,2,\dots,n-1\}$ such that $l_t$ is not adjacent to $p_{n-1}$ and $p_n$, and $l_{t+1}$ is not adjacent to $p_n$, allowing us to glue the consecutive labelings of both $P$ and $L$. There are $q+1$ possible indices $t$ such that $l_t$ is adjacent to $p_{n-1}$, $q+1$ possible indices $t$ such that $l_t$ is adjacent to $p_{n}$, and $q+1$ possible indices $t$ such that $l_{t+1}$ is adjacent to $p_{n}$. Hence there are at most $3(q+1)<(q+1)(q^2+1) - 1 = n-1$ such indices that violate these conditions and so there exists an index $t \in \{1,2,\dots,n-1\}$ such that $l_t$ is not adjacent to $p_{n-1}$ and $p_n$, and $l_{t+1}$ is not adjacent to $p_n$.
    
    We define the following labeling:
    \begin{table}[H]
        \begin{adjustbox}{width=0.85\columnwidth,center}
        \begin{tabular}{|c|c|c|c|c|c||c|c|c|c|c|c|c|}
            \hline
            $v$ & $p_1$ & $p_2$ & $\dots$ & $p_{n-1}$ & $p_n$ & $l_t$ & $l_{t+1}$ & $\dots$ & $l_n$ & $l_1$ & $\dots$ & $l_{t-1}$ \\
            \hline
            $f(v)$ & $1$ & $2$ & $\dots$ & $n-1$ & $n$ & $n+2$ & $n+3$ & $\dots$ & $2n-t+2$ & $2n-t+3$ & $\dots$ & $2n+1$ \\
            \hline
        \end{tabular}
        \end{adjustbox}
    \end{table}
    
    This is indeed a radio labeling since 
    \begin{itemize}
        \item if $|f(u)-f(v)|=1$ then $\{u,v\}$ is $\{p_i,p_{i+1}\}$ or $\{l_i,l_{i+1}\}$, which have distance 4. Hence $|f(u)-f(v)| + d(u,v)=5 = \diam(G)+1$
        \item if $|f(u)-f(v)|=2$ then $\{u,v\}$ is $\{p_i,p_{i+2}\}$, or $\{l_i,l_{i+2}\}$, or $\{p_n,l_t\}$. We have $d(p_i,p_{i+2}) = d(l_i,l_{i+2}) = 4$ and $d(p_n,l_t)=3$ since $p_n$ is not neighbor of $l_t$. Hence $|f(u)-f(v)| + d(u,v)\geq 2+3 = 5 = \diam(G)+1$.
        \item if $|f(u)-f(v)|=3$ then $\{u,v\}$ is $\{p_i,p_{i+3}\}$, or $\{l_i,l_{i+3}\}$, or $\{p_n,l_{t+1}\}$, or $\{p_{n-1},l_{t}\}$. We have $d(p_i,p_{i+3}) \geq 2$, $d(l_i,l_{i+3}) \geq 2$, and $d(p_n,l_{t+1})=d(p_{n-1},l_{t})=3$. Hence $|f(u)-f(v)| + d(u,v)\geq 3+2 = 5 = \diam(G)+1$.
        \item if $|f(u)-f(v)|\geq 4$ then $|f(u)-f(v)| + d(u,v)\geq 4+1 = 5 = \diam(G)+1$.
    \end{itemize}
    
    Thus, we have found a radio labeling whose span is $2n+1$, so $rn(G) \le 2n+1$. Therefore, $rn(G) = 2n+1 = 2(q+1)(q^2+1)+1$.

\end{proof}
\begin{rem}
    Unlike the $(q+1,6)-$cage that is radio graceful, the $(q+1,8)-$cage is not radio graceful because of its even diameter. However, we observe that the $(q+1,8)-$cage is almost radio graceful, as the minimum labeling span only expands by $1$ beyond the graceful labeling span.
\end{rem}
\subsubsection{Cages with girth 12}    

The incidence graph $G_I$ of a generalized hexagon of order $q$ is the unique $(q+1,12)-$cage \cite{generalized-polygons}, a bipartite $(q+1)-$regular graph on $2(q^3+1)(q^2+q+1)$ vertices with diameter $6$ and girth $12$. We have the following result on the radio gracefulness of $(q+1,12)-$cage.  

\begin{thm}\label{hex_not_rg}
    Let $G_I$ be the incidence graph of a generalized hexagon of order $q$. Then $G_I$ is not radio graceful.
\end{thm}

The proof of \cref{hex_not_rg} follows from \cref{bip_even} as $G_I$ is bipartite and has diameter $6$.

    The following Proposition is used to show the existence of the $4^{th}$ power of a Hamiltonian cycle in each of the components of $A(G_I)$, a fact that we will find useful in determining the radio number of $G_I$. While we only use it for the $4^{th}$ power of a Hamiltonian cycle, we state it for the general $\ell^{th}$ power of a Hamiltonian cycle.

    \begin{prop}\label{4thpower}
        Let $G$ be a graph with $n$ vertices where $\delta(G) \geq \frac{4\ell-1}{4\ell}n$. Then $G$ contains the $\ell^{th}$ power of a Hamiltonian cycle.
    \end{prop}
    
    \begin{proof}
        Let $G$ be an $n-$vertex graph with minimum degree $d = \delta(G) \ge \frac{4\ell-1}{4\ell}n$. We say a cycle is \emph{special} if it is the $\ell^{th}$ power of a cycle.
        Suppose the longest special cycle has $k$ vertices. We will prove that $k = n$ by contradiction. Assume that $k<n$ and $C=(v_1,v_2,\dots,v_k,v_1)$ is a special cycle. Then any vertex $u$ outside of this cycle is not adjacent to at least one of the $2\ell$ vertices $v_{k-\ell+1},v_{k-\ell+2},\dots,v_k,v_1,v_2,\dots,v_\ell$, otherwise the cycle could be extended contradicting maximality. Each vertex in $G$ has at most $n-1-d$ non-neighbors, so the number of vertices not adjacent to at least one of these $2\ell$ vertices is at most $2\ell(n-1-d)$. Therefore, the number of vertices outside of the cycle $C$ is
        \begin{equation}\label{eq1}
            n-k \leq 2\ell(n-1-d).
        \end{equation}
        
        Let $u$ be a vertex of $G$ that is not part of this cycle. Then for each index $i$, we have $u$ is not adjacent to at least one of the vertices $(v_i,v_{i+1},v_{i+2},\dots,v_{i+2\ell-1})$ (let $v_{n+j} = v_{j}$ when $j>0$), otherwise $(v_1,v_2,\dots,v_{i+2}, v_{i+\ell-1}, u, v_{i+\ell}, v_{i+\ell+1},\dots,v_k,v_1)$ is a special cycle of length $k+1>k$, a contradiction. The number of indices $i$ such that $u$ is not adjacent to $v_i$ is at most $n-1-d$. Similarly for $v_{i+1},v_{i+2},\dots,v_{i+2\ell-1}$, we thus have the number of indices $i$ such that $u$ is not adjacent to at least one of the vertices $(v_i,v_{i+1},v_{i+2},\dots,v_{i+2\ell-1})$ is at most $2\ell(n-1-d)$. Since all indices $i = 1, 2, \dots, k$ have to satisfy this condition, we must have
        \begin{equation}\label{eq2}
            k \leq 2\ell(n-1-d)
        \end{equation}
        From (\ref{eq1}) and (\ref{eq2}), we have $n \leq 4\ell(n-1-d)$, which contradicts $d \geq \frac{4\ell-1}{4\ell}n$.
    \end{proof}

    Let $G_P$ be the graph induced by vertices corresponding to the set of points $P$ in $A(G_I)$ and let $G_L$ be the graph induced by vertices corresponding to the set of points $L$ in $A(G_I)$.

    \begin{lem}\label{gen_hex_anti}
        The antipodal graph $A(G_I)$ of the incidence graph $G_I$ of a generalized hexagon of order $q$ is a $q^5-$regular bipartite graph with two components $G_P$ and $G_L$. For $q > 15$, the graphs $G_P$ and $G_L$ both contain the $4^{th}$ power of a Hamiltonian cycle.
    \end{lem}

    \begin{proof}
        Let $\Gamma_i(v)$ be the set of vertices that are distance $i$ from $v$ in $G_I$. Each point $v$ has exactly $q+1$ neighbors, so $|\Gamma_1(v)| = q+1$. Each neighbor of $v$ has $q$ neighbors other than $v$, and no two neighbors $u_1,u_2$ of $v$ have the same neighbor $w \neq v$, otherwise resulting in two irreducible paths of length $2$ from $v$ to $w$ in $G_I$. Therefore, $|\Gamma_2(v)|= (q+1)q$. Similarly, we must have $|\Gamma_3(v)|= (q+1)q^2$, $|\Gamma_4(v)|= (q+1)q^3$, and $|\Gamma_5(v)|= (q+1)q^4$. Then, $$|\Gamma_6(v)|= 2(q^3+1)(q^2+q+1) - (1 + |\Gamma_1(v)|+ |\Gamma_2(v)|+ |\Gamma_3(v)|+ |\Gamma_4(v)|+ |\Gamma_5(v)|) = q^5.$$
        Hence, each vertex $v$ in $G$ has $q^5$ vertices distance 6 from it, so $A(G_I)$ is a $q^5-$regular graph with two components corresponding to the set of points $P$ and the set of lines $L$. Since $q^5 \geq \frac{15}{16}(q^3+1)(q^2+q+1)$ for $q > 15$, \cref{4thpower} for $\ell=4$ shows the existence of the $4^{th}$ power of a Hamiltonian cycle in both $G_P$ and $G_L$.
    \end{proof}

\begin{thm}
    Let $G$ be the $(q+1,12)-$cage where $q > 15$. Then $rn(G) = 2(q^3+1)(q^2+q+1)+1$.
\end{thm}

\begin{proof}
   Since $G$ is the incidence graph of a generalized hexagon of order $q$, we can divide its vertex set into 2 parts $P$ and $L$, corresponding to the set of points and the set of lines, respectively. Let $n = |P|=|L| = (q^3+1)(q^2+q+1)$
   
Since $G$ is not radio graceful by \cref{hex_not_rg}, we have $rn(G) \geq |V(G)|+1 = 2(q^3+1)(q^2+q+1)+1 = 2n+1$. We will prove that this lower bound can occur. Following the same argument as in the proof of \cref{rn-quad}, we give a radio labeling $f: V(G) \rightarrow \{1,2,\dots, 2n+1\}$ by showing the existence of a consecutive labeling for both $P$ and $L$ then a gluing strategy, all of which satisfy the radio labeling condition (see \cref{radio_condition}). It is enough to find an ordering of the vertices of $P = \{p_1,p_2,\dots,p_n\}$ such that $d(p_i,p_{i+1}) = 6$ for $i = 1,2,\dots,n-1$, $d(p_i,p_{i+2}) = 6$ for $i = 1,2,\dots,n-2$, $d(p_i,p_{i+3}) = 6$ for $i = 1,2,\dots,n-3$, and $d(p_i,p_{i+4}) = 6$ for $i = 1,2,\dots,n-4$. Similarly for the vertices of $L=\{l_1, \dots, l_n\}$. 

Consider the subgraphs $G_P$ induced by the set of points $P$ and $G_L$ induced by the set of lines $L$. Our goal is to find a Hamiltonian path in both $G_P$ and $G_L$ such that any vertex is adjacent to every vertex within distance $4$ along the cycle, or the $4^{th}$ power of a Hamiltonian cycle. By \cref{gen_hex_anti}, both $G_P$ and $G_L$ both contain the $4^{th}$ power of a Hamiltonian cycle, say $(p_1,p_2,\dots,p_n,p_1)$ and $(l_1,l_2,\dots,l_n,l_1)$, respectively. Thus, there exists a consecutive labeling for both $P$ and $L$ that satisfies the radio labeling condition.

    Hence there are at most $5(q+1)<(q^3+1)(q^2+q+1) - 1 = n-1$ such indices that violate these conditions and so there exists an index $t \in \{1,2,\dots,n-1\}$ such that $l_t$ is not adjacent to $p_{n-1}$ and $p_n$,
    
    In the graph $G$, by a similar argument as in the proof of \cref{rn-quad}, there are at most $5(q+1)<(q^3+1)(q^2+q+1) - 1 = n-1$ (as $q>2$) indices $t$ such that $l_t$ is not adjacent to $p_{n-1}$ and $p_n$. Thus, there exists an index $t \in \{1,2,\dots,n-1\}$ such that $d(l_t,p_n) \geq 5$, $d(l_t,p_{n-1}) \geq 4$, $d(l_t,p_{n-2}) \geq 3$ , $d(l_t,p_{n-3}) \geq 2$, $d(l_{t+1}, p_{n})\geq 4$, $d(l_{t+1}, p_{n-1})\geq 3$, $d(l_{t+1}, p_{n-2})\geq 2$, $d(l_{t+2}, p_{n})\geq 3$, $d(l_{t+2}, p_{n-1})\geq 2$, and $d(l_{t+3}, p_{n})\geq 2$. Then we define the following radio labeling:
    \begin{table}[H]
        \begin{adjustbox}{width=0.85\columnwidth,center}
        \begin{tabular}{|c|c|c|c|c|c||c|c|c|c|c|c|c|}
            \hline
            $v$ & $p_1$ & $p_2$ & $\dots$ & $p_{n-1}$ & $p_n$ & $l_t$ & $l_{t+1}$ & $\dots$ & $l_n$ & $l_1$ & $\dots$ & $l_{t-1}$ \\
            \hline
            $f(v)$ & $1$ & $2$ & $\dots$ & $n-1$ & $n$ & $n+2$ & $n+3$ & $\dots$ & $2n-t+2$ & $2n-t+3$ & $\dots$ & $2n+1$ \\
            \hline
        \end{tabular}
        \end{adjustbox}
    \end{table}
    
    Thus, we have found a radio labeling whose span is $2n+1$, so $rn(G) \le 2n+1$. Therefore, $rn(G) = 2n+1 = 2(q^3+1)(q^2+q+1)+1$.
\end{proof}

\subsection{Approximate-Moore graphs and their complements}\label{ER_q}

The radio gracefulness of well-known approximate Moore graphs is now considered. Throughout this section, let $q$ be a prime power. The Erd\H{o}s-Rényi polarity graph (or Brown graph) \( ER_q \) is defined as follows.

\begin{defn}
Let \( PG(2,q) \) denote the projective plane over \( \mathbb{F}_q \), whose points are equivalence classes of nonzero vectors \( (x,y,z) \in \mathbb{F}_q^3 \setminus \{0\} \) under scalar multiplication. Define \( ER_q \) to be the graph whose vertex set is the set of all points of \( PG(2,q) \) where an edge between two vertices \( (x_0,y_0,z_0) \) and \( (x_1,y_1,z_1) \) exists if and only if $x_0x_1 + y_0y_1 + z_0z_1 = 0 \text{ in } \mathbb{F}_q$.
\end{defn}

We now recall key structural properties of these graphs.

\begin{lem}[Bachraty and Siran \cite{siran}]
\label{erq-properties}
Let \( ER_q \) be the Erd\H{o}s-Rényi polarity graph of order \( q \). Then:
\begin{enumerate}
    \item \( |V(ER_q)| = q^2 + q + 1 \),
    \item The vertex set can be partitioned into \( q+1 \) \emph{quadric} vertices of degree \( q \), and the remaining \( q^2 \) vertices of degree \( q+1 \),
    \item \( ER_q \) has diameter 2,
    \item For adjacent vertices \( u \sim v \), there is at most one common neighbor.
\end{enumerate}
\end{lem}

We use these facts to prove Erd\H{o}s-Rényi polarity graphs are radio graceful.

\begin{lem}
\label{erq-ham-path}
For all $q \ge 2$, the graph \( \overline{ER_q} \) has a Hamiltonian path.
\end{lem}

\begin{proof}
Since \( ER_q \) has maximum degree \( q+1 \), every vertex in \( \overline{ER_q} \) has a degree at least $q^2 + q + 1 - 1 - (q+1) = q^2 - 1$. The inequality $q^2 - 1 \geq \frac{q^2 + q}{2}$ holds for all \( q \geq 2 \) so by \cref{Dirac}, it follows that \( \overline{ER_q} \) has a Hamiltonian path.

\end{proof}

\begin{thm}
\label{erq-radio}
For all \( q \geq 2 \), the Erd\H{o}s-Rényi polarity graph \( ER_q \) is radio graceful.
\end{thm}

\begin{proof}
By \cref{erq-properties}, \( \operatorname{diam}(ER_q) = 2 \). By \cref{erq-ham-path}, \( \overline{ER_q} \) has a Hamiltonian path. From \cref{Antipodal_2}, we have \( ER_q \) is radio graceful.
\end{proof}

\cref{erq-radio} is an example of the following result which holds more generally for diameter $2$ graphs with bounded degree, which follows from \cref{Antipodal_2} and \cref{Dirac}.

\begin{porism}
\label{porism:diam2}
Let $G$ be a graph of diameter $2$ with $n$ vertices. If $\Delta(G) \leq \frac{n-1}{2}$, then $G$ is radio graceful.
\end{porism}

We can apply this result to show that McKay-Miller-\v{S}irá\v{n} graphs (or MMS graphs) are radio graceful. We first give a definition of the MMS graph $H_q$.

\begin{defn}
    Let $q>2$, and let $\xi$ be a primitive element of the finite field $\mathbb{F}_q$. Let $X = \{1, \xi^2,\dots\}$ and $X' = \{\xi, \xi^3, \dots\}$ be the subsets of $\mathbb{F}_q$ as defined in \cite{HAFNER}. Let $H_q$ be the MMS graph. The vertex set of $H_q$ is $\mathbb{Z}_2 \times \mathbb{F}_q \times \mathbb{F}_q$ with edges between given by:

    \begin{itemize}
        \item $(0,x,y)$ is adjacent to $(0,x,y')$ if and only if $y-y' \in X$;
        \item $(1,m,c)$ is adjacent to $(1,m,c')$ if and only if $c-c' \in X'$;
        \item $(0,x,y)$ is adjacent to $(1,m,c)$ if and only if $y=mx+c$.
    \end{itemize}
\end{defn}

\begin{lem}[Mckay, Miller, \v{S}irá\v{n} \cite{MMS}]\label{MMS_lemma}
    Every MMS graph corresponds to an odd integer $d \equiv 1 \pmod{6}$ such that $\frac{2d+1}{3}$ is prime. The corresponding MMS graph is of degree $d$ and order
    \[
    n = \frac{8}{9} \left(d + \frac{1}{2} \right)^2.
    \]
\end{lem}

\begin{cor}[of \cref{porism:diam2}]
Let $G$ be a MMS graph. Then $G$ is radio graceful.
\end{cor}

\begin{proof}
By \cref{porism:diam2}, as MMS graphs have diameter 2, to show $G$ is radio graceful it suffices to show that
\[
d \leq \frac{n - 1}{2}.
\]
Equivalently,
\[
d + \frac{1}{2} \leq \frac{n}{2} = \frac{4}{9} \left(d + \frac{1}{2} \right)^2,
\]
where the equality follows from \cref{MMS_lemma}. This inequality holds for all $d \geq 7$, and so it holds when $d \equiv 1 \pmod{6}$ is an odd integer such that $\frac{2d+1}{3}$ is prime.
\end{proof}

A sort of converse of \cref{porism:diam2} is also known:

\begin{thm}[Saha and Basunia \cite{saha2020}]
    \label{complement}
    Let G be an n-vertex graph with $\Delta(G) < \frac{n}{2}$ . If G is traceable, then the complement graph $\overline{G}$ is always radio graceful. 
\end{thm}

Let $T_{m,n}$ be the tadpole graph, the graph constructed by adding an edge between an endpoint of $P_n$ and a vertex in $C_m$. From \cref{complement}, we can obtain three infinite families of radio graceful graphs from low-degree traceable graphs.
    
\begin{example}
    The graphs $\overline{C_m}$, $\overline{P_n}$, and $\overline{T_{m,n}}$ are radio graceful for $m\geq5$, $n\geq 5$, and $m+n\geq7$, respectively.
\end{example}

One natural corollary of \cref{complement} is as follows:

\begin{cor}
    For $n \ge 5$, let $G$ be an $\frac{n-1}{2}-$regular graph on $n$ vertices. Then $G$ and $\overline{G}$ are radio graceful.
\end{cor}

\begin{proof}
    First, note that the complement of an $\frac{n-1}{2}-$regular graph is also an $\frac{n-1}{2}-$regular graph. By \cref{Dirac} both have Hamiltonian paths. Hence, by \cref{complement} their complements, which are each other, are both radio graceful.
\end{proof}

We wish to apply \cref{complement} to polarity graphs. To show that polarity graphs are traceable, we introduce an equivalent construction using Singer difference sets. A \emph{Singer difference set} $D$ is a set of $q+1$ elements of $\mathbb{Z}_{q^2+q+1}$ such that the set $\{(d_i - d_j) \pmod{q^2+q+1} \ | \ d_i, d_j \in D \text{ and } d_i \neq d_j \}$ is the set of all integers from $1$ to $q^2+q$ without repetition. 

\begin{defn}\cite{singer}\label{Sq_def}
    Let $D = \{d_0, d_1, \dots, d_{q+1}\}$ be the Singer difference set $D$ over $\mathbb{Z}_{q^2+q+1}$. The Singer graph $S_q$ has vertices $V = \{i \ | \ 0  \le i < q^2+q+1\}$, and edges $E = \{(i, j) \ | \ (i + j) \pmod{q^2+q+1} \in D\}$. 
\end{defn}

\begin{lem}[Lakhotia, Isham, Monroe, Besta, Hoefler, Petrini \cite{polarfly}]
    \label{polarfly}
    For all $q \ge 2$, the Singer graph $S_q$ is Hamiltonian. 
\end{lem}

\begin{lem}[Erskine, Fratri\v{c}, and \v{S}irá\v{n} \cite{singer}]\label{singer_iso}
    Let $q$ be any prime power. Then the graphs $S_q$ and $ER_q$ are isomorphic.
\end{lem}

\begin{thm}\label{ERQcomplement}
    $\overline{ER_q}$ is radio graceful.
\end{thm}

\begin{proof}
    From \cref{polarfly} and \cref{singer_iso} it follows that $ER_q$ has a Hamiltonian path. For $q\geq 2$, we have $\Delta(ER_q)=q+1<\frac{q^2+q}2=\frac{|V(ER_q)|-1}{2}$. By \cref{complement}, it follows that $\overline{ER_q}$ is radio graceful.
\end{proof}

\subsubsection{Graceful Radio Labeling of $ER_q$ and $\overline{ER_q}$}

A graceful radio labeling of $ER_q$ and $\overline{ER_q}$ using the construction of Hamiltonian cycles in Singer graphs described in \cite{polarfly} is now given.

\begin{construction}
Choose $d_0$ and $d_1$ from the Singer difference $D$ and choose $j_0, j_1 \in \mathbb{N}$ such that $d_0-j_0 \notin D$, $d_1-j_1 \notin D$, and $\gcd((d_0-j_0)-(d_1-j_1),q^2+q+1)=1$. A Hamiltonian path $v_1,v_2,\ldots,v_{q^2+q+1}$ in $\overline{S_q} \cong \overline{ER_q}$ can be constructed by setting $v_1=\frac{q^2+q+2}{2}d_1-1$, setting $v_i=d_0-j_0-v_{i-1}$ if $i$ is even, and setting $v_i=d_1-j_1-v_{i-1}$ if $i$ is odd. To generate the radio graceful labeling of $ER_q$, give each vertex $v_i$ the label $i$.
\end{construction}

Note, by \cref{Sq_def}, that $(v_{i-1}, v_{i}) \notin E(ER_q)$ as $v_{i-1} + v_{i}=d_0-j_0$ if $i$ is even where $d_0-j_0 \notin D$ and $v_{i-1} + v_{i}=d_1-j_1$ if $i$ is odd where $d_1-j_1 \notin D$.

\begin{construction}
Choose $d_0$ and $d_1$ from the Singer difference $D$ such that $\gcd(d_0-d_1,q^2+q+1)=1$. A Hamiltonian path $v_1,v_2,\ldots,v_{q^2+q+1}$ in $S_q \cong ER_q$ can be constructed by setting $v_1=\frac{q^2+q+2}{2}d_1$, setting $v_i=d_0-v_{i-1}$ if $i$ is even, and setting $v_i=d_1-v_{i-1}$ if $i$ is odd $\cite{polarfly}$. To generate the radio graceful labeling of $\overline{ER_q}$ give each vertex $v_i$ the label $i$.
\end{construction}

\section{Conclusions and Open Problems}\label{sec:open}

We have examined the radio gracefulness of different families of cages, including cages with girth 6, 8, or 12, which can be nicely represented as incidence graphs of generalized polygons. In addition, we considered many other known cages detailed in Figure~\ref{fig:enter-label}. By \cref{Antipodal}, almost all these cages are not radio graceful, with the only exceptions being complete graphs, which are Moore graphs of girth $3$, and Moore graphs of girth $5$ and $6$. Given these facts, we ask the following question.

\begin{figure}[H]
    \centering
    \includegraphics[width=0.5\linewidth]{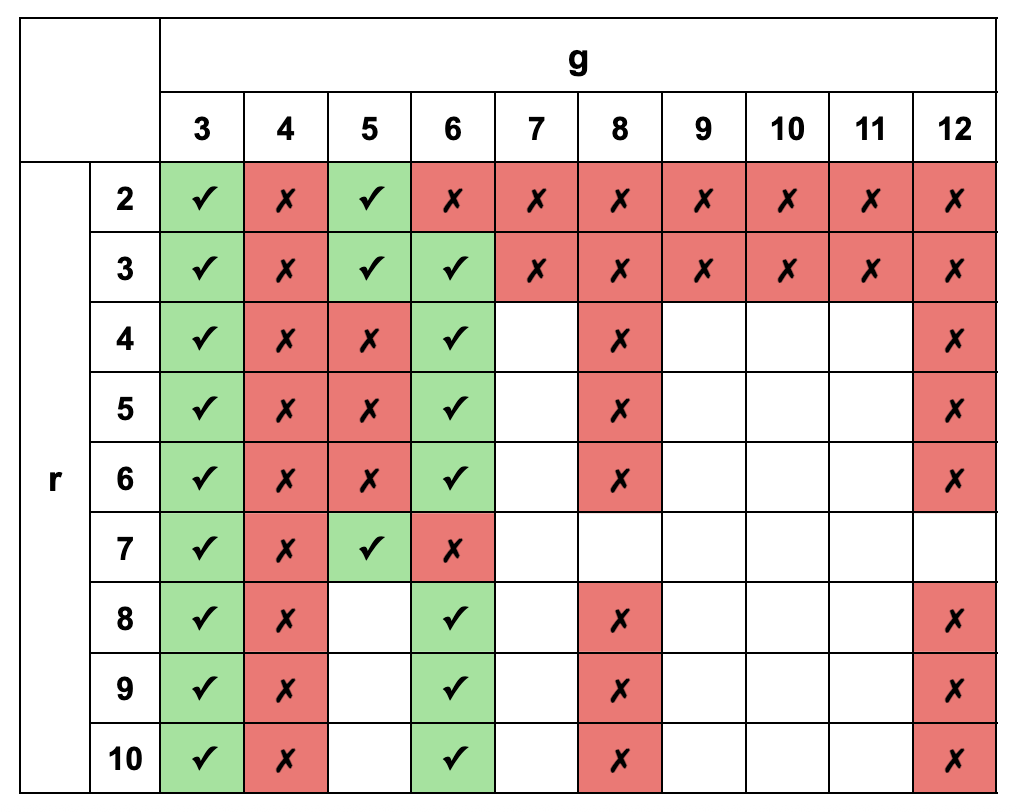}
    \caption{Radio gracefulness of various $(r,g)-$cages}
    \label{fig:enter-label}
\end{figure}

\begin{qu}
    Are the only radio graceful $(r,g)-$cages Moore graphs of girth $3$, $5$, and $6$?
\end{qu}

\section*{Acknowledgments}

This research was conducted in Summer 2025 as part of the Summer Undergraduate Applied Mathematics Institute at Carnegie Mellon University, during which the third and fourth authors were supported by the NSF Grant REU-2244348. The second author is supported by the NSF Grant DMS-2536176.

\bibliographystyle{plainurl}
\bibliography{references}
\end{document}